\documentclass[11pt]{article}
\usepackage{amsthm}
\usepackage{amsmath}
\usepackage{amssymb}
\usepackage{enumerate}
\usepackage{fullpage}
\usepackage[
  colorlinks,
  linkcolor = blue,
  citecolor = blue,
  urlcolor = blue]{hyperref}
\usepackage{subfig}
\newtheorem{theorem}{Theorem}[section]
\newtheorem{lemma}[theorem]{Lemma}
\newtheorem{corollary}[theorem]{Corollary}

\newtheorem*{thm1}{Theorem}

\newtheorem*{thm3}{Theorem}
\newtheorem*{thm4}{Theorem}

\newcommand\cref[1]{Corollary~\ref{cor:#1}}

\DeclareMathOperator{\VP}{VP}

\title{Symmetric Graphs with respect to Graph Entropy}
\author{Seyed Saeed Changiz Rezaei, Ehsan Chiniforooshan}

\begin{document}

\author{Seyed Saeed Changiz Rezaei\\[1mm]
Department of Mathematics\\
Simon Fraser University\\
Burnaby, BC, Canada\\
  {\tt sschangi@sfu.ca}
\and
Ehsan Chiniforooshan\\[1mm]
Google Inc.\\
Waterloo, ON, Canada\\
 {\tt chiniforooshan@google.com}
}

\maketitle

\begin{abstract}
 Let $F_G(P)$ be a functional defined on the set of all the probability distributions on the vertex set of a graph $G$. We say that $G$ is \emph{symmetric with respect to $F_G(P)$} if the uniform distribution on $V(G)$ maximizes $F_G(P)$. Using the combinatorial definition of the entropy of a graph in terms of its vertex packing polytope and the relationship between the graph entropy and fractional chromatic number, we characterize all  graphs which are symmetric with respect to graph entropy. We show that a graph is symmetric with respect to graph entropy if and only if its vertex set can be uniformly covered by its maximum size independent sets. Furthermore, given any strictly positive probability distribution $P$ on the vertex set of a graph $G$, we show that  $P$ is a maximizer of the entropy of graph $G$ if and only if its vertex set can be uniformly covered by its maximum weighted independent sets. We also show that the problem of deciding if a graph is symmetric with respect to graph entropy, where the weight of the vertices is given by probability distribution $P$, is co-NP-hard.
\end{abstract}

\section{Introduction}
The entropy of a graph is an information theoretic functional which is defined on a graph with a probability distribution on its vertex set. This functional was originally proposed by J. K\"{o}rner in 1973 to study the minimum number of codewords required for representing an information source~\cite{JKor}.

Let $VP(G)$ be the \emph{vertex packing polytope} of a given graph $G$ which is the convex hull of the characteristic vectors of its independent sets. Let $ n := |V(G)|$ and $P$ be a probability distribution on $V(G)$. Then the \emph{entropy of $G$ with respect to the probability distribution $P$} is defined as
\[
H(G,P) = \min_{\mathbf{a}\in VP(G)} \sum_{v\in V(G)} p_v\log (1/a_v).\label{eq:combent}
\]

J. K\"{o}rner investigated the basic properties of the graph entropy in several papers from 1973 till 1992~\cite{JKor,JKor01,JKor11,JKor1,JKor3,Jkor4,JKor2}.

Let $F$ and $G$ be two graphs on the same vertex set $V$. Then the union of graphs $F$ and $G$ is the graph $F\cup G$ with vertex set $V$ and its edge set is the union of the edge set of graph $F$ and the edge set of graph $G$. That is
\begin{eqnarray}
&&V\left(F\cup G\right) = V,\nonumber\\
&&E\left(F\cup G\right) = E\left(F\right)\cup E\left(G\right).\nonumber
\end{eqnarray}
Perhaps, the most important property of the entropy of a graph is that it is sub-additive with respect to the union of graphs, that is
\[
H\left(F\cup G,P\right)\leq H\left(F,P\right) + H\left(G,P\right).
\]
This leads to the application of graph entropy for graph covering problem as well as the problem of perfect hashing.

The graph covering problem can be described as follows. Given a graph $G$ and a family of graphs $\mathcal G$ where each graph $G_i\in \mathcal G$ has the same vertex set as $G$, we want to cover the edge set of $G$ with the minimum number of graphs from $\mathcal G$. Using the sub-additivity of graph entropy one can obtain lower bounds on this number.

Graph entropy was used implicitly in a paper by Fredman and Koml\'{o}s for the minimum number of perfect hash functions of a given range that hash all $k$-element subsets of a set of a given size~\cite{FK}.

Simonyi showed that the maximum of the graph entropy of a given graph over the probability distribution on its vertex set is equal to the logarithm of its fractional chromatic number~\cite{Simu}. In this paper, we characterize all strictly positive probability distributions which maximize the entropy of a given graph. 

Let $\mathcal S$ be a multi-set of independent sets of a graph $G$. We say $\mathcal S$ is \emph{uniform} over a subset of vertices $W$ of the vertex set of $G$ if each vertex $v\in W$ is covered by a constant number of independent sets in $\mathcal S$.
Then, our main result can be stated as follows.
\begin{thm1}
For every graph $G$ and every probability distribution $P$ over $V(G)$, we have $H(G, P) = \lg\chi_f(G[\{v\in V(G):p_v>0\}])$
if and only if there exists a multi-set of independent sets 
  $\cal S$ such that
  \begin{enumerate}
    \item $\cal S$ is uniform over $\{v\in V(G):p_v > 0\}$, and
    \item every independent set $I\in\mathcal{S}$ is a maximum weighted independent
          sets with respect to $P$.
   \end{enumerate}
\end{thm1}

We say a graph is \emph{symmetric with respect to graph entropy} if the uniform probability distribution maximizes its entropy. It is worth noting that the notion of a symmetric graph with respect to a functional was already defined by G. Greco~\cite{Greco}. Furthermore, S.S. C. Rezaei and C. Godsil studied some classes of graphs which are symmetric with respect to graph entropy~\cite{Saeed,Saeed-Chris}. A corollary of the above-mentioned theorem is the following characterization for symmetric graphs.

\begin{thm3}
A graph $G$ is symmetric if and only if $\chi_f(G) = {n\over\alpha(G)}$.
\end{thm3}

Finally we consider the complexity of deciding whether a graph is symmetric with respect to its entropy by proving the following theorem.

\begin{thm4}
It is co-NP-hard to decide whether a given graph $G$ is symmetric.
\end{thm4}

\section{Preliminaries}
Here we recall some properties of entropy of graphs. The following lemma shows the monotonicity for graph entropy.  
\begin{lemma}\emph{(J. K\"{o}rner).}\label{lem:mono}
Let $F$ be a spanning subgraph of a graph $G$. Then for any probability distribution $P$ we have $H(F,P)\leq H(G,P)$.
\end{lemma}
The notion of substitution is defined as follows. Let $F$ and $G$ be two vertex disjoint graphs and $v$ be a vertex of $G$. We substitute $F$ for $v$ by deleting $v$ and joining all vertices of $F$ to those vertices of $G$ which have been adjacent with $v$. Let $G_{v\leftarrow F}$ be the resulting graph. 

We extend the notion of substitution to distributions. Let $P$ and $Q$ be the probability distributions on $V(G)$ and $V(F)$, respectively. Then the probability distribution $P_{v\leftarrow Q} $ on $V\left(G_{v\leftarrow F}\right)$ is given by $P_{v\leftarrow Q}(x) = P(x)$ if $x \in V(G) \setminus \{v\}$ and $P_{v\leftarrow Q}(x) = P(x) Q(x)$ if $x \in V(F)$. 

Now we state the following lemma which was proved in J. K\"{o}rner, et. al. \cite{JKor2}.

\begin{lemma}\emph{(J. K\"{o}rner, G. Simonyi, and Zs.~Tuza).}\label{lem:Subs}
Let $F$ and $G$ be two vertex disjoint graphs, $v$ a vertex of $G$, while $P$ and $Q$ are probability distributions on $V(G) $ and $V(F)$, respectively. Then we have
\[
H\left(G_{v\leftarrow F}, P_{v\leftarrow Q}\right) = H\left(G,P\right) + P(v)H\left(F,Q\right).
\]
\end{lemma}
Notice that the entropy of an empty graph (a graph with no edges) is always zero (regardless of the distribution on its vertices).
\section{Graph entropy and fractional chromatic number}
G. Simonyi established the relationship between the entropy of a graph $H(G, P)$ with its fractional chromatic number $\chi_f(G)$ by showing the following \cite{Sim,Simu}.
\[
\max_P H(G,P) = \log\chi_f(G)
\]

Here we characterize strictly positive probability distributions which maximize $H(G, P)$.\\
First we recall a characterization of uniform independent set covers using fractional chromatic number. A $b$-fold coloring of the vertices of a graph $G$ is an assignment of $b$-subsets of a set with $a$ elements such that adjacent vertices get disjoint $b$-subsets. The least $a$ such that $G$ admits a $b$-fold coloring is called the $b$-fold chromatic number of $G$ and is denoted by $\chi_b(G)$.

\begin{theorem}\label{thm:fracChar}(\cite{fgt})
For every graph $G$ and integer $b$
\[
\chi_f(G) \leq \frac{\chi_b(G)}{b}.
\]
Furthormore, there exists an integer that realizes the equality.
\end{theorem}

As a corollary to the above theorem we have 
\begin{corollary}\label{thm:uniformCover}
Let $\mathcal S$ be a multi-set of independent sets of $G$. Then each element of $\mathcal S$ induces a maximum independent set of $G$ with size $\alpha(G)$ and $\mathcal S$ is a uniform cover the vertices of $G$ if and only if $\chi_f(G) = \frac{n}{\alpha(G)}$.
\end{corollary}
\begin{proof}
First, assuming $\chi_f = \frac{n}{\alpha}$, we prove that there exists a uniform independent set cover whose elements are maximum independent sets. From Theorem \ref{thm:fracChar}, there exists a $b$-fold coloring such that $\chi_b(G) = a$ and we have
\begin{equation}\label{eq:foldColoring}
\chi_f(G) = \frac{n}{\alpha} = \frac{a}{b}.
\end{equation}

Now we construct a uniform maximum independent cover $\mathcal S$ as follows. Let $A$ be the set of colors of size $a$ and $B$ be the function that assigns a $b$-subset of $A$ to every vertex. For every $x\in A$, $I_x=\{v\in V(G) : x\in B(v)\}$ is an independent set. So, ${\mathcal S} = \{I_x : x\in A\}$ is a uniform independent set cover of size $a$. Equation (\ref{eq:foldColoring}) tells us that the average size of independent sets in $\mathcal S$ is $\alpha$ and so they all must be maximum independent sets.

Conversely, assume that $G$ admits a uniform maximum independent set cover $\mathcal S$ such that each vertex $v\in V(G)$ lies in exactly $b$ elements of $\mathcal S$, and so $\frac{|{\mathcal S}|}{b} = \frac{n}{\alpha}$. Then, $B(v) = \{S\in{\mathcal S} : v \in S\}$. is a $b$-fold coloring. Let $b'$ be an integer such that $\chi_f(G) = {\chi_{b'}(G)\over b'}$ and $B'$ be the $b'$-fold coloring that achieves this. Then
\[\frac{n}{\alpha}\leq\frac{\chi_{b'}(G)}{b'}=\chi_f(G)\leq\frac{\chi_b(G)}{b}\leq\frac{|{\mathcal S}|}{b}=\frac{n}{\alpha}.\]
The left-most inequality comes from the fact that $\alpha\chi_{b'}(G)$ is an over-estimation of $\sum_{v\in V(G)} |B'(v)| = nb'$.
\end{proof}

We prove our main result in two steps.
\begin{theorem}\label{thm:entropy}
  For every graph $G$ and every probability distribution $P$ over $V(G)$, if
  $H(G, P) = \lg\chi_f(G[\{v\in V(G) | p_v > 0\}])$, then there exists a multi-set of independent sets
  $\cal S$ such that
  \begin{enumerate}
    \item $\cal S$ is uniform over $\{v\in V(G):p_v > 0\}$, and
    \item every independent set $I\in\mathcal{S}$ is a maximum weighted independent
          sets with respect to $P$.
  \end{enumerate}
\end{theorem}
\begin{proof}
  Consider a fractional coloring
  $(\mathcal{I}, w)$ of $G$, where $\cal I$ is the family of independent sets of
  $G$ and $w:\mathcal{I}\to\mathbb{Q}^+$ is a weight function such that
  \begin{equation}\label{eq:frac_min}
    \sum_{I\in\mathcal{I}} w_I = \chi_f(G),
  \end{equation}
  and
  \begin{equation}\label{eq:frac_sum}
    \sum_{I \in\mathcal{I}, v\in I} w_I \geq 1,
  \end{equation} for all $v\in V(G)$. We define ${\bf x}^* = \sum_{I\in\mathcal{I}}
  {w_I\over\chi_f(G)}\cdot{\bf I}$, where $\bf I$ is the characteristic vector of $I$.
  Clearly, ${\bf x}^*\in \VP(G)$ due to (\ref{eq:frac_min}), and $x^*_v \geq
  {1\over\chi_f(G)}$, for all $v\in V(G)$, due to (\ref{eq:frac_sum}).
  
  We turn the set $\cal I$ to a multi-set $\cal S$ by setting the multiplicity
  of any independent set $I$ to $r\cdot w_I$, where $r$ is an integer such that
  $r \cdot w_I\in\mathbb{N}$ for all $I\in\mathcal{I}$. If for some $v\in D_P =
  \{v\in V(G):p_v > 0\}$, $x^*_v > {1\over\chi_f(G)}$, then $H(G,P)\leq
  -\sum_{v\in V(G)}p_v\lg x^*_v< \lg\chi_f(G)$, which is a contradiction. So,
  every $v\in D_P$ is in exactly $r\cdot x^*_v\cdot\chi_f(G)=r$ independent sets
  of $\cal S$, which means $\cal S$ is uniform over $D_P$.
  
  We also claim that $P_I = \alpha_P$ for all $I\in\mathcal{I}$, where
  $\alpha_P$ is the weight of maximum weighted independent set of $G$ with
  respect to $P$. Indeed, if there exists an $I_1\in\mathcal{I}$ such that
  $p_{I_1}<\alpha_p$, then we define ${\bf
  x}^\epsilon = {\bf x}^* - \epsilon {\bf I}_1 + \epsilon {\bf M}$, where $\bf M$
  is the characteristic vector of an arbitrary maximum weighted independent set.
  Note that ${\bf x}^\epsilon\in\VP(G)$ for all $0\leq\epsilon\leq w_{I_1}$.

  Now consider
  \[
  d(\epsilon) = \sum_{v\in I_1\setminus M} p_v\lg {x^\epsilon_v\over x^*_v} + \sum_{v\in M\setminus I_1}p_v\lg {x^\epsilon_v\over x^*_v}
  = \lg (1-\epsilon\chi_f(G))\sum_{v\in I_1\setminus M} p_v + \lg(1 + \epsilon\chi_f(G))\sum_{v\in M\setminus I_1}p_v.
  \]
  We have $d(0) = 0$, $d$ is differentiable at zero, and $d'(0) = (\sum_{v\in M\setminus I_1}p_v - \sum_{v\in I_1\setminus M}p_v) \chi_f(G) > 0$.
  So, there exists a positive number $\epsilon^+$ such that $d(\epsilon)$ is positive for all $\epsilon$ in $(0, \epsilon^+)$.
  Then, for an arbitrary $\epsilon \in (0, \min(w_{I_1}, \epsilon^+))$, $H(G, P)\leq -\sum_{v\in V(G)}p_v\lg x^\epsilon_v=-\sum_{v\in V(G)}p_v\lg x^*_v - d(\epsilon) <-\sum_{v\in V(G)}p_v\lg x^*_v = \lg\chi_f(G)$,
  which is a contradiction.
\end{proof}

It is easy to see that the converse of Theorem~\ref{thm:entropy} holds for uniform distributions. We call a graph $G$ \emph{symmetric} if $H(G, U) \geq H(G, P)$ for all
probability distributions $P$ on $V(G)$, where $U$ is the uniform distribution.

\begin{corollary}\label{thm:symmetric}
A graph $G$ is symmetric if and only if $\chi_f(G) = {n\over\alpha(G)}$.
\end{corollary}
\begin{proof}
  If $\chi_f(G) = {n\over\alpha}$, then
  for every ${\bf x} \in \VP(G)$ we have
  \[\sum_{v\in V(G)} x_v\leq\alpha\Rightarrow \prod_{v\in V(G)} x_v \leq
  \left({\alpha\over n}\right)^n\Rightarrow -{1\over n}\sum_{v\in V(G)}\lg x_v
  \geq -{1\over n}\lg\left({\alpha\over n}\right)^n=-\lg{\alpha\over n} = lg\chi_f(G).\]
  This means $H(G, u) = \lg \chi_f(G)$, and so $G$ is symmetric.
  
  On the other hand, if $G$ is symmetric, then, according to
  Theorem~\ref{thm:entropy}, there exists a uniform maximum independent set cover. Thus from Theorem \ref{thm:uniformCover}, we have $\chi_f(G) = {n\over\alpha}$.
\end{proof}

Now we prove the converse of Theorem $\ref{thm:entropy}$ for every distribution.

\begin{theorem}
Let $G$ be a graph with probability distribution $P$ on its vertex set. Suppose that there exists a multi-set of independent sets
  $\cal S$ such that
  \begin{enumerate}
    \item $\cal S$ is uniform over $\{v\in V(G):p_v > 0\}$, and
    \item every independent set $I\in\mathcal{S}$ is a maximum weighted independent
          sets with respect to $P$.
  \end{enumerate}
  Then $H(G,P) = \log \chi_f(G[\{v\in V(G): p_v > 0\}])$.
\end{theorem}
\begin{proof}
First we assume that the probability of every vertex $v\in V(G)$, i.e., $p_v$ is equal to $\frac{n_v}{m}$ for some $n_v$ and $m\in\mathbb{N}$. We then construct the graph $G^\prime$ by blowing each vertex $v$ up $n_v$ times and making the corresponding vertices adjacent to the neighbours of $v$. We consider each set of vertices substituted for each vertex $v$ as an independent set $F_v$ of size $n_v$ with uniform distribution $\frac{1}{n_v}$. Then $G^\prime$ is a probabilistic graph with uniform distribution $\frac{\mathbf 1}{m}$ on its vertex set, and repeated application of Lemma \ref{lem:Subs} leads to
\[
H(G^\prime,\frac{\mathbf 1}{m}) = H(G,P) + \sum_{v} \frac{n_v}{m} H(F_v,\frac{\mathbf 1}{n_v}).
\]
Noting that $F_v$ is an independent set, we have
\[
H(F_v, \frac{\mathbf 1}{n_v}) = 0,~\forall~v\in V(G),
\]
and consequently, 
\begin{equation}\label{eq:newG}
H(G^\prime,\frac{\mathbf 1}{m}) = H(G,P).
\end{equation}
Note that the vertex set of $G$ is uniformly covered by maximum weighted independent set with respect to $p$. Therefore, due to the construction of $G^\prime$, graph $G^\prime$ is uniformly covered by its maximum independent sets. Thus using Theorem \ref{thm:symmetric} and equation (\ref{eq:newG}), we have
\begin{equation}\label{eq:fract}
H(G^\prime,\frac{\mathbf 1}{m}) = \log \chi_f(G^\prime) = H(G,P) \leq \log\chi_f(G).
\end{equation} 
Noting that there is a homomorphism from $G$ to $G^\prime$, we get
\[
\chi_f(G)\leq\chi_f(G^\prime).
\]
This along with (\ref{eq:fract}) implies that $\chi_f(G) = \chi_f(G^\prime)$, and hence, probability distribution $p$ over the vertex set of $G$ maximizes $H(G,P)$.

Now, suppose that $P$ is a real probability distribution over the vertex set of $G$. Now we show that graph entropy $H(G,P)$ is a continuous function of $P$. Let $P$ be a strictly positive probability distribution over $V(G)$. 
Then for every $\epsilon>0$ we set
\[
\delta = \frac{1}{2} \min_{v\in V} p_v \times\min \left(1,\frac{\epsilon}{|V(G)|H(G,P)}\right).
\]   
Therefore, we have 
\[
\|P - P^\prime\|<\delta \rightarrow \|H(G,P^\prime) - H(G,P)\|<\epsilon.
\]
First we show $H(G,P^\prime)<H(G,P) +\epsilon$. Let $\mathbf{a}^*\in VP(G)$ achieves $H(G,P)$, that is
\[
H(G,P) = \sum_vp_v\log\frac{1}{a^*_v}.
\] 
Thus
\begin{eqnarray}
&&\left(\min_v p_v\right)\times\sum_v\log\frac{1}{a_v^*}\leq H(G,P)\nonumber\\
&&\Rightarrow \sum\log\frac{1}{a^*_v}\leq \frac{H(G,P)}{\min_v p_v}\Rightarrow \log\frac{1}{a^*_v}\leq\frac{H(G,P)}{\min_v p_v},~\forall v\in V.
\end{eqnarray}
On the other hand, setting $\delta_v = p^\prime_v - p_v$, we have
\begin{eqnarray}
H(G,P^\prime) &=& \min_{\mathbf a\in VP(G)}\sum_v(p_v+\delta_v)\log\frac{1}{a_v}\nonumber\\
&\leq &\sum_v(p_v+\delta_v)\log\frac{1}{a^*_v} = H(G,P) + \sum_v\delta_v\log\frac{1}{a^*_v}\nonumber\\
&\leq &H(G,P) + |V(G)|\times\delta\times \frac{H(G,P)}{\min_v p_v} < H(G,P) +\frac{\epsilon}{2}.
\end{eqnarray}
Now we show that $H(G,P^\prime)>H(G,P) - \epsilon$. Let $\mathbf{b}^*\in VP(G)$ achieves $H(G,P^\prime)$, that is
\begin{equation}\label{eq:entropy2}
H(G,P^\prime) = \sum_vp^\prime_v\log\frac{1}{b^*_v}.
\end{equation}
If $H(G,P)\leq H(G^\prime, P)$, we are done. Thus we may assume $H(G,P)>H(G,P^\prime)$. Then using the value for $\delta$ defined above we have
\[
\log\left(\frac{1}{b^*_v}\right)\leq\frac{H(G,P^\prime)}{\min_v p_v^\prime}\leq \frac{2H(G,P)}{\min_v p_v^\prime},~\forall v\in V(G).
\]
Now using the above equation and equation (\ref{eq:entropy2}), we get
\begin{eqnarray}
H(G,P^\prime)&=&\sum_vp_v\log\frac{1}{b^*_v}+\sum_v\log\frac{1}{b^*_v}\nonumber\\
&\geq& H(G,P) + \sum -|\delta_i|\log\frac{1}{b^*_v}\nonumber\\
&\geq& H(G,P) - |V(G)|\times 2\times\frac{H(G,P)}{\min_vp_v}\times\delta\nonumber\\
&\geq& H(G,P) - \epsilon.
\end{eqnarray}
Thus $H(G,P)$ is a continuous function. Now note that there exists a sequence of  rational probability distribution $P_k$ which tends to $P$ as $k\rightarrow \infty$. Consequently, there exists a sequence of graphs $G^\prime_k$ with the corresponding sequence of uniform probability distributions $\frac{\mathbf 1}{m_k}$ constructed as explained above. Then, noting that $H(G,P)$ is a continuous function with respect to $P$ and using (\ref{eq:fract}), we have
\begin{equation}\label{eq:limit1}
\lim_{k\rightarrow \infty }H(G^\prime_k,\frac{\mathbf 1}{m_k}) =\lim_{k\rightarrow \infty } \log \chi_f(G^\prime_k) = \lim_{k\rightarrow \infty }H(G,P_k) = H(G,P) \leq \log\chi_f(G).
\end{equation}  
Since there is a homomorphism from $G$ to $G^\prime_k$ for every $k$, we get
\begin{equation}\label{eq:limit2}
\chi_f(G)\leq\lim_{k\rightarrow \infty } \chi_f(G^\prime_k).
\end{equation}
Therefore, (\ref{eq:limit1}) and (\ref{eq:limit2}) imply that probability distribution $P$ over the vertex set of $G$ maximizes $H(G,P)$.
\end{proof}
\subsection{Computational Complexity}
In this section, we discuss the complexity of computing graph entropy by proving the following theorem.
\begin{theorem}
  It is co-NP-hard to decide whether a given graph $G$ is symmetric.
\end{theorem}
\begin{proof}
 We use a reduction from
  the maximum independent set problem. Assume we are given a graph $F$ and an
  integer $k$, and let $A$ and $B$ be two disjoint sets of size $k - 1$ disjoint
  from $V(F)$. Then, graph $F$ has an independent set of size at least $k$ if and only
  if the graph $G$, defined below, is not symmetric:
  \[V(G) = A \cup V(F)\times B,\]
  \begin{eqnarray*}
  E(G) &=& \{\{a, (v, b))\}: a\in A, v\in V(F), b\in B\}\cup\\
       &&  \{\{(v, b), (v', b')\}: v, v'\in V(F), b,b'\in B, v\neq v', b\neq b'\} \cup\\
       &&  \{\{(v, b), (v', b)\}: (v, v')\in E(F), b\in B\}
  \end{eqnarray*}
  To see this, it is enough to note that $\alpha(G) = \max\{\alpha(F), k - 1\}$
  and $G$ has a uniform independent set cover whose independents sets are all of
  size $k - 1$. If $\alpha(F)\leq k -1$, then $\alpha(G) = k - 1$ and due to the construction of $G$ the vertex set of $G$ is covered uniformly by its maximum size independent sets. Therefore, using Theorem \ref{thm:symmetric}, graph $G$ is symmetric with respect to graph entropy. Conversely, if $\alpha(F)\geq k $, then since the vertices of $G$ in $A$ are adjacent to all vertices in $V(G)\setminus A$, the vertex set of $G$ is not covered uniformly by its maximum size independent sets. Consequently, from Theorem \ref{thm:symmetric}, graph $G$ is not symmetric with respect to graph entropy.  
\end{proof}



\end{document}